\documentclass[11pt]{amsart}

\usepackage[margin=1.1in]{geometry}
\usepackage{amsmath,amssymb,amsthm,mathtools}
\usepackage{enumitem}
\usepackage{microtype}
\usepackage[colorlinks=true,linkcolor=blue,citecolor=blue,urlcolor=blue]{hyperref}

\newtheorem{theorem}{Theorem}
\newtheorem{lemma}[theorem]{Lemma}

\newtheorem{remark}[theorem]{Remark}

\DeclareMathOperator{\ordop}{ord}

\newcommand{\ZZ}{\mathbb Z}
\newcommand{\NN}{\mathbb N}
\newcommand{\RR}{\mathbb R}

\title{Averaged Fourier Estimates and Dyadic Approximation on the Cantor set}
\author{Prasuna Bandi}
\date{}

\address{Department of Mathematics, Indian Institute of Science Education and Research Thiruvananthapuram, Vithura, Thiruvananthapuram, Kerala 695551, India}

\email{prasuna@iisertvm.ac.in}

\subjclass[2020]{Primary 11J83; Secondary 28A80, 11K60, 37A45}

\keywords{Middle-third Cantor set, dyadic approximation, metric Diophantine approximation, Cantor measure, Fourier transform}

\begin{document}

\begin{abstract}
Let $C$ be the middle-third Cantor set and let $\mu$ be the natural Cantor probability measure. Let
\[
        \gamma=\frac{\log2}{\log3}.
\]
The two main results of this paper are
\[
        \mu\{x\in C:\|2^n x\|<n^{-\tau}\text{ for infinitely many }n\}=0
        \qquad \text{ for } \tau>2-\gamma.
\]
and 
\[
      \mu\{x\in C:\|2^n x\|<n^{-\tau}\text{ for infinitely many }n\}=1 \qquad \text{ for } \tau<\frac{1-\gamma}{2}.
\]
These results give new progress toward Velani's conjecture on
zero-one law for dyadic approximation in the middle-third Cantor set.
\end{abstract}

\maketitle
\section{Introduction}

Let \(C\) denote the middle-third Cantor set and let \(\mu\) be the
natural Cantor probability measure on \(C\). Thus
\[
        C=\left\{x\in[0,1]:
        x=\sum_{k=1}^{\infty}\frac{a_k}{3^k},
        \ a_k\in\{0,2\}\right\},
\]
and
\[
        \gamma=\dim_H C=\frac{\log2}{\log3}.
\]
We write
\[
        \|t\|=\min_{m\in\mathbb Z}|t-m|
\]
for the distance from \(t\) to the nearest integer.\\
For a positive function \(\psi\), define
\[
        W_2(\psi)
        =
        \left\{
        x\in C:
        \left|x-\frac{p}{2^n}\right|
        <
        \frac{\psi(2^n)}{2^n}
        \text{ for infinitely many }(p,n)\in\mathbb Z\times\mathbb N
        \right\}.
\]
Equivalently,
\[
        x\in W_2(\psi)
        \quad\Longleftrightarrow\quad
        \|2^n x\|<\psi(2^n)
        \text{ for infinitely many }n.
\]
This paper concerns dyadic approximation in the middle-third Cantor set, which is the study of the set \(W_2(\psi)\).  The broader
metric theory of rational approximation on fractal sets goes back to
Mahler's question \cite{Mahler} and has since developed in several directions.
Recent progress includes approximation of Cantor points by arbitrary
rationals \cite{Benardhezhang,DattaJana}, intrinsic approximation by rational
points lying in the Cantor set itself \cite{Tanwangwu,zbMATH07301410,zbMATH06077915}, and triadic
approximation, where the allowed denominators are powers of \(3\)
\cite{Levsalpvelani}.  
In the
triadic setting, where denominators are powers of \(3\), the geometry of
the Cantor construction is aligned with the denominators, and a
zero-full law follows from the work of Levesley, Salp and Velani
\cite{Levsalpvelani}.  The dyadic problem is more arithmetic, the denominators \(2^n\) are multiplicatively independent
from the construction base \(3\) which makes it more difficult. This is also related to Furstenberg's $\times 2\times 3$ principle; see the introduction of \cite{Baker} for further discussion.

A conjecture of Velani, stated in \cite{AllenChowYu}, predicts the following dichotomy for the set $W_2(\psi)$.  For monotone \(\psi\),
\[
        \mu(W_2(\psi))
        =
        \begin{cases}
        0, & \displaystyle \sum_{n=1}^{\infty}\psi(2^n)<\infty,\\[0.7em]
        1, & \displaystyle \sum_{n=1}^{\infty}\psi(2^n)=\infty.
        \end{cases}
\]
For the power functions
\[
        \psi(2^n)=n^{-\tau},
\]
this conjecture states
\[
        \mu\{x\in C:\|2^n x\|<n^{-\tau}\text{ infinitely often}\}=0
        \quad \text{ for } \tau>1,
\]
and
\[
        \mu\{x\in C:\|2^n x\|<n^{-\tau}\text{ infinitely often}\}=1
        \quad \text{ for }\tau\leq 1.
\]

The purpose of the present paper is to improve the known 
ranges on both sides.  Our first result is
a convergence theorem.

\begin{theorem}\label{thm:intro-convergence}
Let $\tau>2-\gamma.$
Then
\[
        \mu\left(
        \left\{
        x\in C:\|2^n x\|<n^{-\tau}
        \text{ for infinitely many }n
        \right\}
        \right)=0.
\]
\end{theorem}

Since
\[
        2-\gamma
        =
        2-\frac{\log2}{\log3}
        \approx 1.36907,
\]
this gives a convergence result closer to the conjectural threshold
\(\tau=1\) than the previously known range of Allen--Baker--Chow--Yu
\cite{AllenBakerChowYuNote}, who proved the corresponding 
statement for
\[
        \tau>\frac1\gamma-0.03
        =
        \frac{\log3}{\log2}-0.03
        \approx 1.55496.
\]
Our second result is a divergence theorem.

\begin{theorem}\label{thm:intro-divergence}
Let
\[
        0<\tau<\frac{1-\gamma}{2}.
\]
Then
\[
        \mu\left(
        \left\{
        x\in C:\|2^n x\|<n^{-\tau}
        \text{ for infinitely many }n
        \right\}
        \right)=1.
\]
\end{theorem}

Numerically,
\[
        \frac{1-\gamma}{2}
        =
        \frac12\left(1-\frac{\log2}{\log3}\right)
        \approx 0.184535.
\]
This improves the exponent \(0.01\) obtained by Baker
\cite{Baker}.\\\\
The methods of this paper can also be adapted to obtain inhomogeneous versions
and asymptotic counting statements analogous to Baker's theorem, but we restrict attention here to the homogeneous zero-one statements.\\\\
We now describe the main input of the paper.  Let
\[
        e(t)=e^{2\pi i t}.
\]
The Fourier transform of the Cantor measure has the classical product
formula
\[
        \widehat\mu(q)
        =
        \prod_{r=1}^{\infty}
        \frac{1+e(2q/3^r)}{2},
        \qquad \text{ for s} q\in\mathbb Z.
\]
Hence,
\[
        |\widehat\mu(q)|
        =
        \prod_{r=1}^{\infty}
        \left|
        \cos\left(\frac{2\pi q}{3^r}\right)
        \right|.
\]
Our main technical input is an average estimate for this product along
the orbit \(q2^n\).  If \(H\ge3\) and \(K\) is chosen by
\[
        3^K\le H<3^{K+1},
\]
then we prove that, for every nonzero integer \(q\),
\[
        \sum_{M<n\le M+H}
        |\widehat\mu(q2^n)|
        \ll
        H^\gamma
        3^{(1-\gamma)\min(\nu_3(q),K)}.
\]
The trivial bound is \(H\), since \(|\widehat\mu|\le1\).  Thus, when
\(q\) is not highly divisible by \(3\), this estimate gives an improvement over the trivial bound.\\
We also prove the bilinear estimate, for nonzero integers
\(u,v\),
\[
        \sum_{M<n,m\le M+H}
        |\widehat\mu(u2^n+v2^m)|
        \ll
        H^{1+\gamma}
        3^{(1-\gamma)\min(\nu_3(u),\nu_3(v),K)}.
\]
These estimates are proved by combining the exact order of \(2\) modulo
powers of \(3\) with a finite averaging argument. The general strategy is inspired by Schmidt's treatment of
Fourier-product sums in his work on normal numbers 
\cite{Schmidt}, but the special pair of bases \((2,3)\) gives a more
direct and improved estimates.

After this work had been completed, we became aware
of the preprint of Dai, Li, Wang and Wu
\cite{dai2026metricresultsdyadicapproximation}, which also studies the same problem. The arguments and results in the
present paper were obtained independently.

\section{Averaged Fourier estimates}
For $q\in\ZZ$, the Fourier transform of $\mu$ at $q$ is given by 
\[
        \widehat\mu(q)=\int e(qx)\,d\mu(x) =\prod_{r=1}^{\infty}\frac{1+e(2q/3^r)}2=e(q/2)\prod_{r=1}^{\infty}\cos\left(\frac{2\pi q}{3^r}\right).
\]
Therefore
\begin{equation*}
        |\widehat\mu(q)|
        =\prod_{r=1}^{\infty}
        \left|\cos\left(\frac{2\pi q}{3^r}\right)\right|.
\end{equation*}
For $K\ge0$, define
\[
        P_K(q):=\prod_{r=1}^{K}
        \left|\cos\left(\frac{2\pi q}{3^r}\right)\right|,
        \qquad \text{ and }P_0(q):=1.
\]
Then
\begin{equation}\label{eq:muhatlepk}
        |\widehat\mu(q)|\le P_K(q)
\end{equation}
It is easy to see that if $q_1=q_2 \bmod 3^K$, then $P_K(q_1)=P_K(q_2)$. Hence $P_K(q)$ depends only on $q\bmod 3^K$.
\begin{lemma}\label{lem:cos}
For every $\theta\in\RR$,
\begin{equation}\label{eq:cos}
        \sum_{k=0}^{2}
        \left|\cos\left(\theta+\frac{2\pi k}{3}\right)\right|
        \le 2.
\end{equation}
\end{lemma}

\begin{proof}
Let
\[
        A_k=\cos\left(\theta+\frac{2\pi k}{3}\right),
        \qquad k=0,1,2.
\]
Then $A_0+A_1+A_2=0$ and $|A_k|\le1$.  Since the sum is zero, the sum of the positive terms is equal to sum of the absolute values of the negative terms and this sum is at most $1$ because any one term has absolute value at most $1$. Hence
\[
        |A_0|+|A_1|+|A_2|\le2
\]
and this proves \eqref{eq:cos}.
\end{proof}
\begin{lemma}\label{lem:induction}
For every $a,L\ge0$ and for every $b\in\ZZ$,
\begin{equation}\label{eq:induction}
        \sum_{0\le t<3^L}P_{a+L}(b+3^a t)
        \le 2^L.
\end{equation}   
\end{lemma}
\begin{proof}
We prove \eqref{eq:induction} by induction on $L$. If $L=0$, it is clear.  Now, suppose $L\ge1$ and write $t$ as
\[
        t=t_0+3^{L-1}\varepsilon,
        \qquad \text{ where } 0\le t_0<3^{L-1},\quad \text{ and }\varepsilon\in\{0,1,2\}.
\]
Then
\begin{align*}
   P_{a+L}(b+3^a t)&=P_{a+L}(b+3^at_0 +3^{a+L-1})\\
   &=P_{a+L-1}(b+3^at_0 +3^{a+L-1})\left|
        \cos\left(
        \frac{2\pi(b+3^a t_0)}{3^{a+L}}
        +\frac{2\pi\varepsilon}{3}
        \right)
        \right| \\
        &=P_{a+L-1}(b+3^a t_0)\left|
        \cos\left(
        \frac{2\pi(b+3^a t_0)}{3^{a+L}}
        +\frac{2\pi\varepsilon}{3}
        \right)
        \right|.
\end{align*}
where the last equality follows from the fact that $P_K(q)$ depends only on $q \bmod 3^K$.\\
Therefore,
\begin{align*}
    \sum_{0\le t<3^L}P_{a+L}(b+3^a t)&=\sum_{0\le t_0<3^{L-1}} P_{a+L-1}(b+3^a t_0)\sum_{\varepsilon=0}^2\left|
        \cos\left(
        \frac{2\pi(b+3^a t_0)}{3^{a+L}}
        +\frac{2\pi\varepsilon}{3}
        \right)
        \right|\\
        &\overset{\eqref{eq:cos}}{\leq} 2 \sum_{0\le t_0<3^{L-1}} P_{a+L-1}(b+3^a t_0) \leq 2^L 
\end{align*}
where the last inequality follows by induction hypothesis. This proves \eqref{eq:induction}.   
\end{proof}
We thank Simon Baker and Nikita Shulga for pointing out that one may
group \(j\ge2\) ternary digits at a time in the cosine-tree argument,
replacing the factor \(2^j\) by a smaller \(j\)-block constant.  This
would slightly improve the numerical exponents, but we keep the
one-digit version to retain the simple closed form
\(\gamma=\log2/\log3\).
\begin{remark}
We thank Simon Baker and Nikita Shulga for pointing out that the
above Lemma can be improved slightly by grouping \(j\ge2\) digits at a time instead of one. i.e. write
\[
        t=t_0+3^{L-j}\varepsilon,
        \qquad
        0\le t_0<3^{L-j},\quad 0\le \varepsilon<3^j.
\]
and estimate the contribution of the last $j$ factors in \(P_{a+L}\). This gives a small improvement in the numerical exponents. For example, when $j=2$ the one-digit bound \(2^L\) can be improved to approximately $2^L(0.9848)^L$. We keep the one digit estimate in order to retain the simple form.
\end{remark}
The remaining estimates in this section are inspired by Schmidt's method in his
work on normal numbers.  In particular, Schmidt used
order estimates modulo prime powers and residue-class multiplicity
bounds to control Fourier-product sums, see Lemma~4 and its Corollary,
and Lemmas~5--7 of \cite{Schmidt}.  In the present problem the
special pair \((r,s)=(2,3)\) gives a sharper and more elementary method. We also replace Schmidt's digit-pair estimates (Lemma 2 of \cite{Schmidt}) by a direct averaging
argument using Lemma \ref{lem:induction}.\\
Throughout, we write \(\beta:=1-\gamma\).
For a nonzero integer $q$, let
\[
        \nu_3(q)=\max\{a\ge0:3^a\mid q\}.
\]
Denote by \(\ordop_{3^r}(2)\) the multiplicative order of \(2\)
modulo \(3^r\), namely the least positive integer \(n\) for which
\[
        2^n\equiv 1 \pmod{3^r}.
\]
The following Lemma is a standard fact: see, for example,
~\cite[Chapter 10,
Theorem 10.6(a)]{Apostol1976}
\begin{lemma}\label{lem:order}
For every $r\ge1$, 
\[
        \ordop_{3^r}(2)=\varphi(3^r)=2\cdot 3^{r-1}.
\]
\end{lemma}

\begin{lemma}\label{lem:single}
Let $H\ge3$, and choose $K$ such that
\[
        3^K\le H<3^{K+1}.
\]
Then, for every $M\ge0$ and $q\in\ZZ\setminus\{0\}$,
\begin{equation}\label{eq:single}
        \sum_{M<n\le M+H}|\widehat\mu(q2^n)|
        \ll H^\gamma 3^{\beta\min(\nu_3(q),K)}.
\end{equation}
The implicit constant is absolute.
\end{lemma}

\begin{proof}
By \eqref{eq:muhatlepk}, it is enough to estimate
\[
        \sum_{M<n\le M+H}P_K(q2^n).
\]
Put $a=\min(\nu_3(q),K)$.\\
\textbf{Case 1}: $a=K$.\\
$$H=H^{\gamma}H^{\beta}<H^{\gamma}3^{(K+1)\beta}\ll H^{\gamma}3^{K\beta}$$
Since $P_K(q2^n)\leq 1$, we have
\[
        \sum_{M<n\le M+H}P_K(q2^n)\le H \ll H^{\gamma}3^{K\beta}=H^{\gamma}3^{\beta a}
\]
so \eqref{eq:single} follows in this case.\\
\textbf{Case 2}: $a<K$.\\
Write \(q=3^a q_0\), where \(3\nmid q_0\). 
By Lemma \ref{lem:order},
\[
        \operatorname{ord}_{3^{K-a}}(2)
        =
        \varphi(3^{K-a})
        =
        2\cdot 3^{K-a-1}.
\]
Therefore
\[
        \bigl\{
        2^n \bmod 3^{K-a}:
        0\le n<\varphi(3^{K-a})
        \bigr\}
        =
        (\mathbb Z/3^{K-a}\mathbb Z)^\times.
\]
Since \(3\nmid q_0\), multiplication by \(q_0\) is a bijection of the
unit group \((\mathbb Z/3^{K-a}\mathbb Z)^\times\).  
Therefore
\begin{equation}\label{eq:order}
    \bigl\{
        q_0 2^n \bmod 3^{K-a}:
        0\le n<\varphi(3^{K-a})
        \bigr\}
        =
        (\mathbb Z/3^{K-a}\mathbb Z)^\times.
\end{equation}
For $y\in (\mathbb Z/3^{K-a}\mathbb Z)^\times$, define
\[
        \mathcal N(y)
        :=
        \left\{
        n\in\mathbb Z:
        M<n\le M+H,\ 
        q_0 2^n\equiv y \pmod {3^{K-a}}
        \right\}.
\]
Then, by \eqref{eq:order}
\[
       \# \mathcal N(y)
        \le
        1+\frac{H}{2 \cdot 3^{K-a-1}} \leq  1+\frac{3^{K+1}}{2 \cdot 3^{K-a-1}}\ll 3^a
\]
Now, for $n\in \mathcal N(y)$, we have
\[
        q2^n
        =
        3^a q_0 2^n
        \equiv
        3^a y
        \pmod {3^K}.
\]
Hence,
\[
        P_K(q2^n)=P_K(3^a y).
\]
Thus, grouping the sum according to the residue class
\(y=q_0 2^n\bmod 3^{K-a}\), we obtain
\[
\begin{aligned}
        \sum_{M<n\le M+H}P_K(q2^n)
        &=\sum_{y\in (\mathbb Z/3^{K-a}\mathbb Z)^\times} \sum_{n\in \mathcal N(y)}P_K(q2^n)\\
        &\leq \sum_{y\in (\mathbb Z/3^{K-a}\mathbb Z)^\times} \# \mathcal N(y)P_K(3^a y)\\
        &\ll
        3^a
        \sum_{0\le y<3^{K-a}}
        P_K(3^a y).
\end{aligned}
\]
By Lemma \ref{lem:cos}, with $b=0$ and $L=K-a$,
\[
        \sum_{0\le y<3^{K-a}}P_K(3^a y)
        \le 2^{K-a}.
\]
Thus
\[
        \sum_{M<n\le M+H}P_K(q2^n)
        \ll 3^a2^{K-a}
        =2^K\left(\frac32\right)^a.
\]
Since $2^K \ll H^\gamma$ and
$
        \left(\frac32\right)^a=3^{\beta a},
$
we obtain \eqref{eq:single}.
\end{proof}
\begin{lemma}[Bilinear Fourier estimate]\label{lem:bilinear}
Let $H\ge3$, and choose $K$ by $3^K\le H<3^{K+1}$.  Then, for every $M\ge0$ and nonzero integers $u,v$,
\begin{equation}\label{eq:bilinear}
        \sum_{M<n,m\le M+H}
        |\widehat\mu(u2^n+v2^m)|
        \ll
        H^{1+\gamma}
        3^{\beta\min(\nu_3(u),\nu_3(v),K)}.
\end{equation}
The implicit constant is absolute.
\end{lemma}

\begin{proof}
Again we use $|\widehat\mu|\le P_K$.  Let
\[
        a=\min(\nu_3(u),\nu_3(v),K).
\]
If $a=K$, the trivial bound gives $H^2$, while
\[
        H^{1+\gamma}3^{\beta K}\asymp H^{1+\gamma}H^\beta=H^2.
\]
and hence in this case \eqref{eq:bilinear} follows.\\
Now assume $a<K$. WLOG, suppose that $a=\nu_3(v)\le\nu_3(u)$.  Write $v=3^a v_0$ with $3\nmid v_0$.

Fix $n$. By the same argument as in previous lemma, we see that as $m$ varies, $v_0 2^m$ runs through the unit group modulo $3^{K-a}$ with period $2\cdot3^{K-a-1}$.  Since $H\asymp3^K$, each unit residue occurs $O(3^a)$ times. Therefore
\[
\begin{aligned}
        \sum_{M<m\le M+H}P_K(u2^n+v2^m)
        &\ll 3^a
        \sum_{0\le y<3^{K-a}}P_K(u2^n+3^a y).
\end{aligned}
\]
By Lemma \ref{lem:cos}, with $b=u2^n$ and $L=K-a$, we have
\[
        \sum_{0\le y<3^{K-a}}P_K(u2^n+3^a y)
        \le 2^{K-a}.
\]
Hence for each fixed $n$,
\[
        \sum_{M<m\le M+H}P_K(u2^n+v2^m)
        \ll 3^a2^{K-a}.
\]
Therefore,
\[
        \sum_{M<n,m\le M+H}P_K(u2^n+v2^m)
        \ll H3^a2^{K-a}
        =H2^K\left(\frac32\right)^a.
\]
Since $2^K\asymp H^\gamma$ and $(3/2)^a=3^{\beta a}$, this proves \eqref{eq:bilinear}.
\end{proof}

\section{Smooth approximation and coefficient sums}
\begin{lemma}\label{lem:smooth-bump}
Let \(0<R<1/4\),  
and
\[
        I_R=\{t\in\mathbb T:\|t\|<R\}.
\]
For every \(0<\varepsilon<1\), there exists smooth functions
\[
        g_{R}^{-},g_{R}^{+}\in C^\infty(\mathbb T)
\]
such that
\[
        0\le g_{R}^{-}(t)\le \chi_{I_R}(t)
        \le g_{R}^{+}(t)
        \qquad(t\in\mathbb T).
\]
The Fourier coefficients of the functions $g_{R}^{\pm}$ denoted by
$$a_{\ell,R}^{\pm}
        :=
        \int_{\mathbb T}g_{R}^{\pm}(t)e(-\ell t)\,dt$$
satisfy the following:
\begin{equation}\label{eq:a0R}
    a_{0,R}^{-}=c_-R,
        \qquad
        a_{0,R}^{+}=c_+R,
\end{equation}
where \(c_-,c_+>0\) are fixed constants independent of \(R\).
Moreover, for every \(A>1\),
\begin{equation}\label{eq:alR}
   |a_{\ell,R}^{\pm}|
        \ll_{A} R(1+R|\ell|)^{-A}
        \qquad(\ell\in\mathbb Z) 
\end{equation}
\end{lemma}

\begin{proof}

Choose functions \(\phi^-,\phi^+\in C_c^\infty(\mathbb R)\) such that
\[
        0\le \phi^-\le1,
        \qquad
        \operatorname{supp}\phi^-\subset(-1,1),     
\]
and
\[
        \phi^+\ge0,
        \qquad
        \phi^+(u)\ge1\quad(|u|\le1),
        \qquad
        \operatorname{supp}\phi^+\subset(-2,2),
\]
Define
\[
   c_-:=\int_{\mathbb R}\phi^-(u)\,du>0,  \qquad    c_+:=\int_{\mathbb R}\phi^+(u)\,du<\infty.
\]
For \(0<R<1/4\), define functions on \(\mathbb T\) by periodization:
\[
        g_{R}^{\pm}(t)
        =
        \sum_{k\in\mathbb Z}
        \phi^{\pm}\!\left(\frac{t+k}{R}\right).
\]
It follows from the above properties of $\phi^{\pm}$ that
\[
        0\le g_{R}^{-}(t)\le \chi_{I_R}(t)
        \le g_{R}^{+}(t)
        \qquad(t\in\mathbb T),
\]
We now compute the Fourier coefficients.
\begin{align*}
    a_{\ell,R}^{\pm}
        &=
        \int_{\mathbb T}g_{R}^{\pm}(t)e(-\ell t)\,dt\\
        &=\int_{\mathbb R}
        \phi^{\pm}\!\left(\frac{t}{R}\right)e(-\ell t)\,dt\\
        &=   
        R 
        \int_{\mathbb R}\phi^{\pm}(u)e(-R\ell u)\,du\\
        &=
        R \widehat{\phi^{\pm}}(R\ell)
\end{align*}
where
\[
        \widehat{\phi^{\pm}}(\xi)
        =
        \int_{\mathbb R}\phi^{\pm}(u)e(-\xi u)\,du.
\]
In particular,
\[
        a_{0,R}^{-}
        =
        R\int_{\mathbb R}\phi^{-}(u)\,du=c_- R
\]
and
\[
        a_{0,R}^{+}
        =
        R\int_{\mathbb R}\phi^{+}(u)\,du=c_+ R.   
\]

Since \(\phi^{\pm}\in C_c^\infty(\mathbb R)\), for every \(A>1\),
\[
        |\widehat{\phi^{\pm}}(\xi)|\ll_A (1+|\xi|)^{-A}.
\]
Therefore
\[
        |a_{\ell,R}^{\pm}|\ll_A
        R(1+R|\ell|)^{-A}.
\]
\end{proof}

\begin{lemma}(Coefficient sums) \label{lem: coeff sums}
The Fourier coefficients $a_{l,R}^{\pm}$ of $g_R^{\pm}$ from the above lemma satisfy the following inequalities:
\begin{equation}\label{eq:alR single}
    \sum_{\ell\ne0}
        |a_{\ell,R}^{\pm}|\,3^{\beta\nu_3(\ell)}
        \ll 1, 
\end{equation}
    
and
\begin{equation}\label{eq:alR double}
     \sum_{\ell,j\ne0}
        |a_{\ell,R}^{\pm}|\,|a_{j,R}^{\pm}|\,
        3^{\beta\min(\nu_3(\ell),\nu_3(j))}
        \ll 1.
\end{equation}
The implicit constants are absolute.
\end{lemma}
\begin{proof}
   Fix $A=2$, then by $\eqref{eq:alR}$ we have
   $$|a_{\ell,R}^{\pm}|\ll R(1+R|\ell|)^{-2}$$
   Therefore,
\[
\begin{aligned}
        \sum_{\ell\ne0}
        |a_{\ell,R,y}^{\pm}|\,3^{\beta\nu_3(\ell)}
        &\ll
        R
        \sum_{\ell\ne0}
        (1+R|\ell|)^{-2}3^{\beta\nu_3(\ell)}.
\end{aligned}
\]
Write \(\ell=3^a m\), where \(a\ge0\) and \(3\nmid m\).  Then
\[
\begin{aligned}
        R
        \sum_{\ell\ne0}
        (1+R|\ell|)^{-A}3^{\beta\nu_3(\ell)}
        &\le
        R
        \sum_{a\ge0}3^{\beta a}
        \sum_{m\ne0}(1+R3^a|m|)^{-2}.
\end{aligned}
\] 
Since for every $B>0$, 
\[
        \sum_{m\ne0}(1+B|m|)^{-2}\ll B^{-1},
\]
Taking \(B=R3^a\), we get
\[
\begin{aligned}
        R
        \sum_{a\ge0}3^{\beta a}
        \sum_{m\ne0}(1+R3^a|m|)^{-2}
        &\ll
        R
        \sum_{a\ge0}3^{\beta a}(R3^a)^{-1}  \\
        &=
        \sum_{a\ge0}3^{(\beta-1)a}.
\end{aligned}
\]
Since \(\beta=1-\gamma<1\),
\[
        \sum_{a\ge0}3^{(\beta-1)a}
        =
        \sum_{a\ge0}3^{-\gamma a}
        <\infty.
\]
Thus
\[
        \sum_{\ell\ne0}
        |a_{\ell,R}^{\pm}|\,3^{\beta\nu_3(\ell)}
        \ll 1.
\]

For the double sum, we use the inequality
\[
        3^{\beta\min(\nu_3(\ell),\nu_3(j))}\leq 
        \sum_{\substack{a\ge0\\3^a\mid \ell,\ 3^a\mid j}}
        3^{\beta a}.
\]
Therefore
\[
\begin{aligned}
        \sum_{\ell,j\ne0}
        |a_{\ell,R}^{\pm}|\,|a_{j,R}^{\pm}|\,
        3^{\beta\min(\nu_3(\ell),\nu_3(j))}                       \leq 
        \sum_{a\ge0}3^{\beta a}
        \left(
        \sum_{\substack{\ell\ne0\\3^a\mid \ell}}
        |a_{\ell,R}^{\pm}|
        \right)^2.
\end{aligned}
\]
For each \(a\ge0\),
\[
\begin{aligned}
        \sum_{\substack{\ell\ne0\\3^a\mid \ell}}
        |a_{\ell,R}^{\pm}|
        &\ll
        R\sum_{m\ne0}(1+R3^a|m|)^{-2}       \\
        &\ll
        R(R3^a)^{-1}
        =
        3^{-a}.
\end{aligned}
\]
Thus
\[
\begin{aligned}
        \sum_{\ell,j\ne0}
        |a_{\ell,R}^{\pm}|\,|a_{j,R}^{\pm}|\,
        3^{\beta\min(\nu_3(\ell),\nu_3(j))}                       \ll 
        \sum_{a\ge0}3^{\beta a}3^{-2a}
        =
        \sum_{a\ge0}3^{-(2-\beta)a}.
\end{aligned}
\]
Since \(2-\beta=1+\gamma>0\), this sum converges.  Hence
\[
        \sum_{\ell,j\ne0}
        |a_{\ell,R,y}^{\pm}|\,|a_{j,R,y}^{\pm}|\,
        3^{\beta\min(\nu_3(\ell),\nu_3(j))}
        \ll 1.
\]
\end{proof}


\section{Convergence estimates}
For \(r>0\), define
\[
        A_n(r)
        =
        \{x\in C:\|2^n x\|<r\}.
\]
\begin{lemma}[Coarse first moment]\label{lem:coarse-first-moment}
Let $N\ge3$,
and put $ \delta_N=N^{-\beta}$. Then
\begin{equation}\label{eq:coarse-first-moment}
        \sum_{N<n\le2N}\mu(A_n(\delta_N))\ll  N^\gamma.
\end{equation}
\end{lemma}
\begin{proof}

Let \(g_{\delta_N}^{+}\) be the smooth approximation function to $\chi_{I_{\delta_N}}$ from above from
Lemma~\ref{lem:smooth-bump}. It satisfies
\[
        \chi_{I_{\delta_N}}(t)\le g_{\delta_N}^{+}(t),
        \qquad \text{where }
        I_{\delta_N}=\{t\in\mathbb T:\|t\|<\delta_N\}.
\]
Writing its Fourier expansion as
\[
        g_{\delta_N}^{+}(t)
        =
        \sum_{\ell\in\mathbb Z}a_{\ell,\delta_N}^{+}e(\ell t).
\]
by \eqref{eq:a0R} and \eqref{eq:alR single}, we have
\[
        a_{0,\delta_N}^{+}\ll \delta_N \quad \text{ and }  \quad \sum_{\ell\ne0}
        |a_{\ell,\delta_N}^{+}|\,3^{\beta\nu_3(\ell)}
        \ll 1.
\]
Therefore, for every $n\in \NN$
\[
\begin{aligned}
        \mu(A_n(\delta_N))
        &=
        \int_C \chi_{I_{\delta_N}}(2^n x)\,d\mu(x)        \\
        &\le
        \int_C g_{\delta_N}^{+}(2^n x)\,d\mu(x)           \\
        &=
        a_{0,\delta_N}^{+}
        +
        \sum_{\ell\ne0}
        a_{\ell,\delta_N}^{+}\widehat\mu(\ell2^n).
\end{aligned}
\]
Hence,
\[
\begin{aligned}
        \sum_{N<n\le2N}\mu(A_n(\delta_N))
        &\le
        N a_{0,\delta_N}^{+}
        +
        \sum_{\ell\ne0}
        |a_{\ell,\delta_N}^{+}|
        \sum_{N<n\le2N}|\widehat\mu(\ell2^n)|.
\end{aligned}
\]
By \eqref{eq:single}, for $l\neq 0$
\[
        \sum_{N<n\le2N}|\widehat\mu(\ell2^n)|
        \ll
        N^\gamma 3^{\beta\nu_3(\ell)}.
\]
Hence
\[
\begin{aligned}
        \sum_{N<n\le2N}\mu(A_n(\delta_N))
        &\ll
        N\delta_N
        +
        N^\gamma
        \sum_{\ell\ne0}
        |a_{\ell,\delta_N}^{+}|3^{\beta\nu_3(\ell)}   \\
        &\ll
        N\delta_N+N^\gamma\\
        &=N^{1-\beta}+N^{\gamma}=2 N^{\gamma}
\end{aligned}
\]
Thus
\begin{equation}
        \sum_{N<n\le2N}\mu(A_n(\delta_N))
        \ll
        N^\gamma.
\end{equation}
\end{proof}
Since the Cantor measure is $\gamma$-Ahlfors regular, there exists constants $0<c<C<\infty$ such that, for every $x\in C$ and every $0<r\le1$,
\begin{equation}\label{eq:ahlfors}
        c r^\gamma\le \mu(B(x,r))\le C r^\gamma.
\end{equation}
We next prove a coarse-to-fine transfer based on the Ahlfors regularity of the Cantor measure. A closely related endpoint counting form of this Lemma appears in Allen--Baker--Chow--Yu \cite[Lemma~7]{AllenBakerChowYuNote}.
\begin{lemma}[Coarse-to-fine transfer]\label{lem:coarse-to-fine}
There is an absolute constant $C_0$ such that the following holds.  Let $n\ge1$, and
\[
        0<\sigma<\delta/20,
        \qquad
        0<\delta<1/20.
\]
Then
\begin{equation}\label{eq:coarse-to-fine}
        \mu(A_n(\sigma))
        \le
        C_0\left(\frac{\sigma}{\delta}\right)^\gamma
        \mu(A_n(\delta)).
\end{equation}
\end{lemma}

\begin{proof}
For $p\in\mathbb Z$, put
\[
        c_p=\frac{p}{2^n},
        \qquad \text{ and }\qquad
        J_p(\sigma)=\left(c_p-\frac{\sigma}{2^n},c_p+\frac{\sigma}{2^n}\right).
\]
Let
\[
        \mathcal P=\{p\in\mathbb Z:J_p(\sigma)\cap C\ne\varnothing\}.
\]
For each $p\in\mathcal P$, choose $x_p\in J_p(\sigma)\cap C$. Then 
$$J_p(\sigma)\subset \left(x_p-\frac{2\sigma}{2^n},x_p+\frac{2\sigma}{2^n}\right)$$
Therefore
$$A_n(\sigma)\subset \bigcup_{p\in \mathcal P}J_p(\sigma)\cap C\subset \bigcup_{p\in \mathcal P}\left(x_p-\frac{2\sigma}{2^n},x_p+\frac{2\sigma}{2^n}\right)\cap C$$
By the upper Ahlfors bound,
\begin{equation}\label{eq:small-union-upper}
        \mu(A_n(\sigma))
        \ll \#\mathcal P\left(\frac{\sigma}{2^n}\right)^\gamma.
\end{equation}
Now set
\[
        \rho=\frac{\delta}{20\cdot2^n}.
\]
If $p\ne q$, then $|c_p-c_q|\ge2^{-n}$, while
\[
        |x_p-c_p|<\frac{\sigma}{2^n},
        \qquad
        |x_q-c_q|<\frac{\sigma}{2^n}.
\]
Since $\sigma<\delta/20<1/400$, the points $x_p$ are separated by at least $(1-2\sigma)2^{-n}$. Therefore, the balls $B(x_p,\rho)$ are pairwise disjoint. And, if $z\in B(x_p,\rho)$, then
\[
        |z-c_p|\le |z-x_p|+|x_p-c_p|
        <\frac{\delta}{20\cdot2^n}+\frac{\sigma}{2^n}
        <\frac{\delta}{2^n}.
\]
Thus $B(x_p,\rho_*)\cap C\subset A_n(\delta)$.  By the lower Ahlfors bound and disjointness,
\begin{equation}\label{eq:large-union-lower}
        \mu(A_n(\delta))
        \ge \sum_{p\in\mathcal P}\mu(B(x_p,\rho_*))
        \gg \#\mathcal P\left(\frac{\delta}{2^n}\right)^\gamma.
\end{equation}
Combining \eqref{eq:small-union-upper} and \eqref{eq:large-union-lower} gives \eqref{eq:coarse-to-fine}.
\end{proof}
\subsection{Proof of the convergence theorem}
\begin{proof}[Proof of Theorem \ref{thm:intro-convergence}]
Let $\tau>2-\gamma=1+\beta$.  By the Borel--Cantelli lemma it is enough to prove
\begin{equation}\label{eq:bc-sum-needed}
        \sum_{n=1}^\infty \mu(A_n(n^{-\tau}))<\infty.
\end{equation}
Set
\[
        \sigma_N=N^{-\tau},
        \qquad
        \delta_N=N^{-\beta}.
\] 
Since $\tau>\beta$, for all sufficiently large $N$ one has $\sigma_N<\delta_N/20$.  
For $N<n\leq 2N$, Lemma \ref{lem:coarse-to-fine} gives
\[
        \mu(A_n(n^{-\tau}))
        \le \mu(A_n(\sigma_N))
        \ll
        \left(\frac{\sigma_N}{\delta_N}\right)^\gamma
        \mu(A_n(2\delta_N)).
\]
Summing over $N<n\leq 2N$ and applying Lemma \ref{lem:coarse-first-moment} yields
\[
\begin{aligned}
        \sum_{N<n\le2N}\mu(A_n(n^{-\tau}))
        &\ll
        \left(\frac{N^{-\tau}}{N^{-\beta}}\right)^\gamma
        N^\gamma \\
        &=N^{-\gamma(\tau-\beta)}N^\gamma
        =N^{\gamma(1+\beta-\tau)}.
\end{aligned}
\]
Since $1+\beta=2-\gamma$ and $\tau>2-\gamma$, the exponent $\gamma(1+\beta-\tau)$ is negative.  Therefore
\[
        \sum_{k=1}^\infty
        \sum_{2^k<n\le2^{k+1}}\mu(A_n(n^{-\tau}))
        \ll
        \sum_{k=1}^\infty
        2^{k\gamma(2-\gamma-\tau)}<\infty.
\]
This proves \eqref{eq:bc-sum-needed}, and Borel--Cantelli gives
\[
        \mu(W_2(\psi_\tau))=0.
\]
\end{proof}

\section{Divergence estimates}
By Lemma \ref{lem:smooth-bump} and \ref{lem: coeff sums}, for
each \(0<R<1/4\), there is a smooth function \(g_R^{-}:\mathbb T\to\mathbb R\)
such that
\[
        0\le g_R^-(t)\le \chi_{I_R}(t),
\]
and, writing
\[
        g_R^-(t)=\sum_{\ell\in\mathbb Z}a^-_{\ell,R}e(\ell t),
\]
we have
\[
        a^-_{0,R}\gg R, \qquad   |a^-_{\ell,R}|\ll_A R(1+R|\ell|)^{-A}
\]
for every $A>1$, and
\begin{equation}\label{eq:coeff-single-use}
        \sum_{\ell\ne0}|a_{\ell,R}^-|3^{\beta\nu_3(\ell)}\ll1,
\end{equation}
\begin{equation}\label{eq:coeff-double-use}
        \sum_{\ell,j\ne0}
        |a_{\ell,R}^-|\,|a_{j,R}^-|\,
        3^{\beta\min(\nu_3(\ell),\nu_3(j))}
        \ll1.
\end{equation}
Let $R_N=(2N)^{-\tau}$, and set
\[
        Y_N(x)=\sum_{N<n\le2N}g^-_{R_N}(2^n x).
\]
We first estimate
\[
        M_N:=\int_C Y_N(x)\,d\mu(x).
\]
\begin{lemma}\label{lem:mean lower bound}
   For $N$ large enough, 
   \begin{equation}\label{eq:mean lower bound}
       M_N\gg N^{1-\tau}
   \end{equation}
\end{lemma}
\begin{proof}
Using the Fourier expansion of \(g_{R_N}^-\), we have
\[
\begin{aligned}
        M_N
        &=
        \sum_{N<n\le2N}
        \sum_{\ell\in\mathbb Z}
        a_{\ell,R_N}^-\widehat\mu(\ell2^n)      \\
        &=
        Na^-_{0,R_N}
        +
        \sum_{\ell\ne0}
        a^-_{\ell,R_N}
        \sum_{N<n\le2N}\widehat\mu(\ell2^n).
\end{aligned}
\]
The first term satisfies
\[
        Na^-_{0,R_N}\gg NR_N \asymp N^{1-\tau}
\]
and
\[
\begin{aligned}
        \left|
        \sum_{\ell\ne0}
        a^-_{\ell,R_N}
        \sum_{N<n\le2N}\widehat\mu(\ell2^n)
        \right|
        &\le
        \sum_{\ell\ne0}|a^-_{\ell,R_N}|
        \sum_{N<n\le2N}|\widehat\mu(\ell2^n)|        \\
        &\overset{\eqref{eq:single}}{\ll}
        N^\gamma
        \sum_{\ell\ne0}|a^-_{\ell,R_N}|3^{\beta\nu_3(\ell)} \\
        &\overset{\eqref{eq:alR single}}{\ll}
        N^\gamma.
\end{aligned}
\]
Since
\[
        \tau<\frac{1-\gamma}{2}<1-\gamma,
\]
we have
\[
        N^\gamma=o(NR_N).
\]
Thus, for all sufficiently large \(N\),
\begin{equation*}
        M_N\gg N^{1-\tau}.
\end{equation*}
\end{proof}
\begin{lemma}
    \begin{equation}\label{eq:second moment}
        \int_C |Y_N-M_N|^2\,d\mu
        \ll
        N^{1+\gamma}.
\end{equation}
\end{lemma}
\begin{proof}
Let
\[
        Z_N(x)
        =
        \sum_{N<n\le2N}
        \sum_{\ell\ne0}
        a_{\ell,R_N}e(\ell2^n x).
\]
Then,
\[
        \int_C |Y_N-M_N|^2\,d\mu
        \le
        \int_C |Z_N|^2\,d\mu.
\]
Expanding the square,
\[
\begin{aligned}
 \int_C |Z_N(x)|^2\,d\mu(x)
       &=
        \sum_{N<n,m\le2N}
        \sum_{\ell,j\ne0}
        a^-_{\ell,R_N}\overline{a^-_{j,R_N}}\,
        \widehat\mu(\ell2^n-j2^m)\\
        &\leq \sum_{\ell,j\ne0}
        |a^-_{\ell,R_N}|\,|a^-_{j,R_N}|
        \sum_{N<n,m\leq 2N}
        |\widehat\mu(\ell2^n-j2^m)|                \\
        &\overset{\eqref{eq:bilinear}}{\ll}
        N^{1+\gamma}
        \sum_{\ell,j\ne0}
        |a^-_{\ell,R_N}|\,|a^-_{j,R_N}|\,
        3^{\beta\min(\nu_3(\ell),\nu_3(j))}        \\
        &\overset{\eqref{eq:alR double}}{\ll}
        N^{1+\gamma}.
\end{aligned}
\]
\end{proof}
\subsection{Proof of divergence theorem}
\begin{proof}[Proof of Theorem \ref{thm:intro-divergence}]
Define
\[
        \mathcal E_N
        =
        \left\{
        x\in C:
        \|2^n x\|\ge R_N
        \text{ for every }N<n\le2N
        \right\}.
\]
If \(x\in\mathcal E_N\), then \(2^n x\notin\{t: \|t\|<R_N\}\) for every
\(N<n\le2N\). Since
\[
        0\le g^{-}_{R_N}\le \chi_{\{t: \|t\|<R_N\}},
\]
we get
\[
        Y_N(x)=0.
\]
Hence
\[
        \mathcal E_N
        \subseteq
        \{|Y_N-M_N|\ge M_N\}.
\]
By Markov's inequality, together with
\eqref{eq:mean lower bound} and
\eqref{eq:second moment} gives
\[
\begin{aligned}
        \mu(\mathcal E_N)
        &\le
        \frac{1}{M_N^2}
        \int_C |Y_N-M_N|^2\,d\mu     \\
        &\ll
        \frac{N^{1+\gamma}}{N^{2(1-\tau)}}\ll
        N^{-1+\gamma+2\tau}.
\end{aligned}
\]
Because
\[
        \tau<\frac{1-\gamma}{2},
\]
we have
\[
        -1+\gamma+2\tau<0.
\]
Therefore
\[
        \sum_{k=1}^{\infty}\mu(\mathcal E_{2^k})=\sum_{k=1}^{\infty}2^{k(-1+\gamma+2\tau)}<\infty.
\]
By the Borel--Cantelli lemma, for \(\mu\)-almost every \(x\in C\), only
finitely many of the events \(\mathcal E_{2^k}\) occur. Hence, for all
sufficiently large \(k\), there exists
\[
        2^k<n\le2^{k+1}
\]
such that
\[
        \|2^n x\|<R_{2^k}=(2^{k+1})^{-\tau}
        \le n^{-\tau}
\]
Thus
\[
        \|2^n x\|<n^{-\tau}
\]
for infinitely many \(n\). This proves the theorem.\\

\section*{Statements and Declarations}

\subsection*{Conflict of interest}
The author has no conflict of interest.

\subsection*{Data availability}
There is no associated data.

\end{proof}
\bibliographystyle{alpha}
\bibliography{references}
\end{document}